\documentclass[11pt,reqno,tbtags]{amsart}
\usepackage{amssymb}
\usepackage{upref}
\usepackage{natbib}
\bibpunct[, ]{[}{]}{;}{n}{,}{,}

\numberwithin{equation}{section}

\allowdisplaybreaks

\newtheorem{theorem}{Theorem}[section]
\newtheorem{lemma}[theorem]{Lemma}

\newtheorem{conj}[theorem]{Conjecture}

\theoremstyle{definition}

\newtheorem{remark}[theorem]{Remark}

\theoremstyle{remark}

\newenvironment{romenumerate}[1][0pt]{% optional argument changes indentation
\addtolength{\leftmargini}{#1}\begin{enumerate}% gives (i), (ii) etc.
 }{\end{enumerate}}

\newcounter{oldenumi}
{\setcounter{oldenumi}{\value{enumi}}
\begin{romenumerate} \setcounter{enumi}{\value{oldenumi}}}
{\end{romenumerate}}

\newcounter{thmenumerate}

\newcounter{romxenumerate}   %less indented than standard.

\newcounter{xenumerate}   %no left indentation; thus wider lines

\newcommand{\refT}[1]{Theorem~\ref{#1}}

\newcommand{\refL}[1]{Lemma~\ref{#1}}
\newcommand{\refR}[1]{Remark~\ref{#1}}
\newcommand{\refS}[1]{Section~\ref{#1}}

\newcommand{\refConj}[1]{Conjecture~\ref{#1}}

\newcommand{\refand}[2]{\ref{#1} and~\ref{#2}}

\newcommand\marginal[1]{\marginpar{\raggedright\parindent=0pt\tiny #1}}

\begingroup
  \divide\count255 by 60
  \multiply\count255 by -60
  \advance\count255 by \time
  \ifnum \count255 < 10 \xdef\klockan{\the\count1.0\the\count255}
  \else\xdef\klockan{\the\count1.\the\count255}\fi
\endgroup

\newcommand\set[1]{\ensuremath{\{#1\}}}

\newcommand\bigpar[1]{\bigl(#1\bigr)}
\newcommand\Bigpar[1]{\Bigl(#1\Bigr)}

\newcommand\lrpar[1]{\left(#1\right)}

\newcommand\abs[1]{|#1|}
\newcommand\bigabs[1]{\bigl|#1\bigr|}
\newcommand\Bigabs[1]{\Bigl|#1\Bigr|}

\newcommand\lrabs[1]{\left|#1\right|}
\def\rompar(#1){\textup(#1\textup)}    % usage: \rompar(...)

\newcommand\Bigparfrac[2]{\Bigpar{\frac{#1}{#2}}}

\def\xexp(#1){e^{#1}}

\newcommand\ie{i.e.\spacefactor=1000}
\newcommand\eg{e.g.\spacefactor=1000}

\newcommand\ii{\mathrm{i}}

\newcommand{\tend}{\longrightarrow}
\newcommand\dto{\overset{\mathrm{d}}{\tend}}

\newcounter{CC}
\newcounter{cc}
\newcommand\E{\operatorname{\mathbb E{}}}
\renewcommand\P{\operatorname{\mathbb P{}}}
\newcommand\Var{\operatorname{Var}}

\newcommand\gd{\delta}

\newcommand\gf{\varphi}

\newcommand\gs{\sigma}
\newcommand\gss{\sigma^2}
\newcommand\gth{\theta}
\newcommand\gthh{\theta^2}
\newcommand\eps{\varepsilon}

\newcommand\cN{\mathcal N}

\def\[#1]{[\![#1]\!]}

\newcommand\qq{^{1/2}}

\newcommand\qqw{^{-1/2}}

\newcommand\qw{^{-1}}
\newcommand\qww{^{-2}}

\renewcommand{\=}{:=}

\newcommand\intoooo{\int_{-\infty}^\infty}

\newcommand\dd{\,\textup{d}}

\newcommand\citetq[2]{\citeauthor{#2} \cite[{\frenchspacing #1}]{#2}} 

\newcommand\aam{{a_1,\dots,a_m}}
\newcommand\aamp{{a_1+\dots+a_m}}
\newcommand\maam{M_{\aam}}
\newcommand\mab{M_{a,b}}
\newcommand\ax{a^*}
\newcommand\Ax{N_*}
\newcommand\nux{^{(\nu)}}
\newcommand\ab{_{a,b}}
\newcommand\awb{_{a-1,b}}
\newcommand\eith{e^{\ii\gth}}
\newcommand\bfa{\mathbf{a}}
\newcommand\gfd[1]{\gf^{(#1)}}
\newcommand\A{B}
\newcommand\REM[1]{{\raggedright\texttt{[#1]}\par\marginal{XXX}}}

\newenvironment{comment}{\setbox0=\vbox\bgroup}{\egroup} %deletes!

\newcommand\urladdrx[1]{{\urladdr{\def~{\Tilde}#1}}}

\renewcommand\Tilde{{\tiny$\sim$}}

\begin{document}
\title[The Mahonian probability distribution on words]
{The Mahonian Probability Distribution
on Words is Asymptotically Normal}
\thanks{Accompanied by Maple package MahonianStat available from
\hfill\break
\texttt{
http://www.math.rutgers.edu/\Tilde zeilberg/mamarim/mamarimhtml/mahon.html}.\\
The work of D. Zeilberger was supported in part
by the United States of America National Science Foundation.
The work of E. R. Canfield was supported in part by the NSA Mathematical
Sciences Program.
}

\date{August 13, 2009} % (typeset \today{} \klockan)} %; revised ...

\author{E. Rodney CANFIELD}
\address{Computer Science Department,
University of Georgia, Athens, GA 30602-7404, USA}
\email{erc [At] cs [Dot] uga [Dot] edu}

\author{Svante Janson}
\address{Department of Mathematics, Uppsala University, PO Box 480,
SE-751~06 Uppsala, Sweden}
\email{svante.janson [At] math [Dot] uu [Dot] se}
\urladdrx{http://www.math.uu.se/~svante/}

\author{Doron ZEILBERGER}
\address{Mathematics Department, Rutgers University
(New Brunswick), Piscataway, NJ 08854, USA}
\email{zeilberg [At] math [Dot] rutgers [Dot] edu}

%\keywords{<keywords>}
\subjclass[2000]{} 

\begin{abstract}  The Mahonian statistic is the number of inversions
in a permutation of a multiset with $a_i$ elements of type $i$,
$1\le i\le m$.  The counting function for this statistic is the
$q$ analog of the multinomial coefficient 
$\binom{a_1+\cdots+a_m}{a_1,\dots a_m}$, 
and the probability generating function is the
normalization of the latter.    We give two proofs that the
distribution
is asymptotically normal.  The first is {\it computer-assisted},
based on the method of moments.  The Maple package {\tt MahonianStat},
available from the webpage of this article, can be used by the
reader to perform experiments and calculations.  Our second proof
uses characteristic functions.  We then take up the study of a local
limit theorem to accompany our central limit theorem.  Here our result
is less general, and we must be content with a conjecture about
further work.  Our local limit theorem permits us to conclude that the
coeffiecients of the $q$-multinomial are log-concave, provided one stays
near the center (where the largest coefficients reside.)
\end{abstract}

\dedicatory{This article is dedicated to Dennis Stanton,  $q$-grandmaster and versatile
unimodaliter (and log-concaviter)}

\maketitle

%\section*{Dedication}\label{S:dedi} 
%This article is dedicated to Dennis Stanton,  $q$-grandmaster and versatile
%unimodaliter (and log-concaviter).

\section{Introduction}\label{S:intro}

The most important discrete probability distribution,
by far, is the \emph{Binomial distribution},
$B(n,p)$ for which we know everything \emph{explicitly},
$\P(X=i)$ (${\binom{n}{i}}p^i(1-p)^{n-i}$),
the probability generating function ($(pt+(1-p))^n$),
the moment generating function ($(pe^t+1-p)^n$), etc. etc.
Most importantly, it is \emph{asymptotically normal}, which
means that the normalized random variable
$$
Z_n=\frac{X_n-np}{\sqrt{np(1-p)}} 
$$
tends to the standard Normal distribution $N(0,1)$,
as $n \rightarrow \infty$.

Another important discrete distribution function is
the {\it Mahonian} distribution, defined on the
set of {\it permutations} on $n$ objects, and describing,
inter-alia, the random variable ``number of inversions''. 
(Recall that an inversion in a permutation 
$\pi_1, \dots, \pi_n$ is a pair $1 \leq i < j \leq n$
such that $\pi_i>\pi_j$).
Let us call this random variable $M_n$.
The probability generating function, due to Netto, is
given {\it explicitly} by:
\begin{equation}
  \label{netto}
F_n(q)=\frac{1}{n!} \prod_{i=1}^{n} {\frac{1-q^i}{1-q}} \quad .
\end{equation}

%\begin{remark}\label{Rnetto}
  The formula \eqref{netto} has a simple probabilistic interpretation
(see Feller's account in \cite[Section X.6]{FellerI}):  If $Y_j$ is the 
  number of  $i$ with $1 \leq i < j$
and $\pi_i>\pi_j$, then 
\begin{equation}
 \label{nettoy} 
M_n=Y_1+\dots+ Y_n, 
\end{equation}
and $Y_1,\dots,Y_n$ are
  independent random variables
and $Y_j$ is uniformly distributed on \set{0,\dots,j-1},
  as is easily seen by constructing $\pi$ by inserting $1,\dots,n$ in
  this order at random positions; thus $Y_j$ has probability
  generating function $(1-q^j)/(j(1-q))$.
It follows from \eqref{netto} or  \eqref{nettoy} by simple
calculations that the Mahonian distribution has mean and variance
\begin{align}
  \E M_n & = \frac{n(n-1)}4,\label{em}\\
  \Var M_n & = \frac{n(n-1)(2n+5)}{72} 
= \frac{2n^3+3n^2-5n}{72}.
\label{varm}
\end{align}

Even though there is no explicit expression for the coefficients
themselves (\ie{} for the exact probabilitity that a permutation
of $n$ objects would have a certain number of inversions), it is a classical
result (see \cite[Section X.6]{FellerI}),
that follows from an extended form of the Central Limit Theorem,
that the normalized version
$$
\frac {M_n-n(n-1)/4}{\sqrt{(2n^3+3n^2-5n)/72} }   \quad,
$$
tends to $N(0,1)$, as $n \rightarrow \infty$.
So this sequence of probability distributions, too, is asymptotically normal.

But what about {\it words}, also known as 
{\it multi-set permutations}?. Permutations on $n$ objects can
be viewed as words in the alphabet $\{1,2, \dots , n\}$, where
each letter shows up {\it exactly} once. But what if we allow
{\it repetitions}? I.e., we consider all words with
$a_1$ occurrences of $1$, $a_2$ occurrences of $2$,
$\dots$, $a_m$ occurrences of $m$. 
(We assume throughout that $m\ge2$ and each $a_j\ge1$.)
We all know that the number
of such words is the multinomial coefficient
\begin{equation*}
\binom{\aamp}{\aam}
\end{equation*}
and many of us also know
that the number of such words with exactly $k$ inversions is
the coefficient of $q^k$ in the $q$-analog of the multinomial
coefficient
\begin{equation}
  \label{qmulti}
\binom{\aamp}{\aam}_q :=
{{[a_1+ \dots +a_m]!} \over
{[a_1]! \cdots [a_m]!}} \quad,
\end{equation}
where $[n]!:=[1][2] \cdots [n]$, and $[n]:=(1-q^n)/(1-q)$;
see \cite[Theorem 3.6]{Andrews}.
Assuming that all words are equally likely (the uniform 
distribution), the probability generating function is thus
\begin{equation}
  \label{F}
F_{a_1, \dots, a_m}(q):=
{
{(\prod_{i=1}^{m} a_i!) \cdot \prod_{i=1}^{a_1+\dots+a_m} (1-q^i)}
\over
{(a_1 + \dots + a_m)! \prod_{j=1}^{m} \prod_{i=1}^{a_j} (1-q^i) }
}
=
\frac{F_{a_1+\dots+a_m}(q)}{F_{a_1}(q)\dotsm F_{a_m}(q)}.
\end{equation}
Indeed, this can be seen as follows. Let $\maam$ denote the number of
inversions in a random word. If we distinguish the $a_i$ occurrences
of $i$ by adding different fractional parts, in random order, the
number of inversions will increase by $Z_i$, say, with the same
distribution as $M_{a_i}$; further $\maam$ and $Z_1,\dots,Z_m$ are
independent. On the other hand, $\maam+Z_1+\dots+Z_m$ has the same
distribution as $M_{a_1+\dots+a_m}$. Hence,
\begin{equation}\label{FF}
  F_{a_1,\dots,a_m}(q) F_{a_1}(q)\dotsm F_{a_m}(q)=
  F_{a_1+\dots+a_m}(q),
\end{equation}
which is \eqref{F}.

By \eqref{F}, we further have the factorization
\begin{equation}\label{F2}
  F_\aam(q)=\prod_{j=2}^m F_{A_{j-1},a_j}(q),
\end{equation}
where $A_j\= a_1+\dots+ a_j$, which reduces the general case to the
two-letter case.

Note that \eqref{F} shows that the distribution of $\maam$ is
invariant if we permute $a_1,\dots,a_m$; a symmetry which is not
obvious from the definition.

\begin{remark}
  \label{Rpart}
 The two-letter case is particularly
interesting, since the unnormalized generating function
$$
\binom{a+b}{a}F\ab(q)={{(1-q^{a+b})(1-q^{a+b-1}) \cdots (1-q^{a+1})}
\over
{(1-q^{b})(1-q^{b-1}) \cdots (1-q^{1})}} 
=\frac{[a+b]!}{[a]!\,[b]!}
\quad,
$$
(the $q$-binomial coefficient in \eqref{qmulti})
is the same as the generating function for the set of 
integer-partitions with largest part $\leq a$ and
$\leq b$ parts, in other words the set
of integer-partitions whose Ferrers diagram lies
inside an $a$ by $b$ rectangle, where the random variable is
the ``number of dots'' (\ie{} the integer being partitioned).
In other words, the number of such partitions of an integer $n$ equal
the number of words of $a$ 1's and $b$ 2's with $n$ inversions.
See \citetq{Section 3.4}{Andrews}.
\end{remark}

It is easy to see that the {\it mean} of $\maam$ is
$$
\mu(a_1, \dots, a_m):=\E\maam=e_2(a_1, \dots, a_m)/2
$$
(here $e_k(a_1, \dots, a_m)$ is
the degree $k$ elementary symmetric function), so
considering the shifted random variable
$\maam-\mu(\aam)$, 
``number of inversions minus the mean'', 
we get that the probability generating function is
\begin{equation}  \label{G}
G_{a_1, \dots, a_m}(q):=
q^{-\mu(\aam)}F_\aam(q)
=
{{F_{a_1, \dots, a_m}(q)} \over {q^{e_2(a_1, \dots , a_k)/2} }}
\end{equation}
By computing $(q(qG)')'$ and plugging-in $q=1$, 
or from \eqref{FF} and \eqref{em}--\eqref{varm},
it is easy to see
that the {\it variance} $\sigma^2:=\Var\maam$ is
\begin{equation}\label{sigma}
\sigma^2= { {(e_1+1)e_2-e_3} \over {12}} \quad.
\end{equation}
(By $\sigma$ we mean $\sigma(a_1, \dots, a_m)$ and
we omit the arguments $(a_1, \dots, a_m)$ from the
$e_i$'s.)

Let $N\=e_1=\aamp$, the length of the random word,
and let $\ax:=\max_j a_j$ and $\Ax:=N-\ax$.

One main result of the present article is:
\begin{theorem}
  \label{T1}
Consider the random variable, $\maam$,
``number of inversions'',
on the (uniform) sample space of words with $a_1$ 1's,
$a_2$ 2's, $\dots$, $a_m$ $m$'s.
For any sequence of sequences 
$(a_1,\dots,a_m)=(a\nux_1,\dots,a\nux_{m\nux})$ such that
$\Ax\=N-\ax\to\infty$, 
the sequence of normalized random variables
$$
X_{a_1, \dots, a_m}=
\frac{\maam-\mu(a_1, \dots, a_m) }{\sigma(a_1, \dots , a_m)} \quad,
$$
tends to the standard normal distribution $\cN(0,1)$, 
as $\nu \rightarrow \infty$.
\end{theorem}

\refT{T1} includes both the case when $m\ge2$ is fixed, and the case
when $m\to\infty$. 
If $m$ is fixed and $a_1\ge a_2\ge\dots\ge a_m$, as may be assumed by
symmetry, then  the condition $\Ax\to\infty$
is equivalent to $a_2\to\infty$.
In the case $m\to\infty$, the assumption $\Ax\to\infty$
is redundant, because $\Ax\ge m-1$. 

\begin{remark}
  The condition $\Ax\to\infty$ is also necessary for asymptotic
  normality, see \refS{Sfine}.
\end{remark}

\begin{comment}
{\bf Theorem}: Consider the random variable, $inv$,
``number of inversions'',
on the (uniform) sample space of words with $a_1t$ 1's,
$a_2t$ 2's, $\dots$, $a_mt$ $n$'s, then
the sequence of normalized random variables
$$
X_{a_1t, \dots, a_mt}(w):=
{{inv(w)-\mu(a_1t, \dots, a_mt) } \over {\sigma(a_1t, \dots , a_mt)}} \quad,
$$
tends to the standard normal distribution $N(0,1)$, for
any {\it fixed} $a_1, \dots, a_m$, as $t \rightarrow \infty$.
\end{comment}

We give a short proof of this result using characteristic functions in
\refS{Spf2}. 
We  give first in \refS{Spf1} 
another proof (at least of a special case)
that is {\it computer-assisted}, using the Maple package
{\tt MahonianStat} available from the webpage of this article:
\\
{\tt
http://www.math.rutgers.edu/\Tilde zeilberg/mamarim/mamarimhtml/mahon.html},
\\
where one can also find sample input and output.
This first proof uses the {\it method of moments}.

We conjecture that \refT{T1} can be refined to a local limit theorem
as follows:

\begin{conj}\label{Conjlocal}
Uniformly for all $\aam$ and all integers $k$,
  \begin{equation}\label{local}
\P(\maam=k)
=\frac1{\sqrt{2\pi}\gs}\Bigpar{e^{-(k-\mu)^2/(2\gs^2)}+O\Bigparfrac1{\Ax}}.	
  \end{equation}
\end{conj}

We have not been able to prove this conjecture in full generality, but
we prove it under additional hypotheses on $\aam$ in \refS{Slocal}.

\section{A computer-inspired proof}\label{Spf1}
%{\bf A Computer-Inspired Proof of The Two-Letter Case}

We assume for simplicity that $m$ is fixed, and that 
$(\aam)=(ta^0_1,\dots,ta^0_m)$ for some fixed $a^0_1,\dots,a^0_m$ and
$t\to\infty$. 

We discover and prove the leading term in the asymptotic 
expansion, in $t$, for
an {\it arbitrary} $2r$-th moment, for the normalized random variable
$X_\aam=(\maam- \mu)/\sigma$, and show that it converges to the
moment $\mu_{2r}=(2r)!/(2^rr!)$ of $\cN(0,1)$, for every $r$.

For the sake of exposition, we will only treat in detail the
two-letter case, where we can find {\it explicit} expressions
for the asymptotics of the $2r$-th moment of $M_{a_1,a_2}-\mu$, for
$a_1=ta$, $a_2=tb$ with
symbolic $a,b,t$ and $r$ to {\it any} desired (specific) order $s$ 
(\ie{} the leading coefficient $t^{3r}$ as well as
the terms involving $t^{3r-1}, \dots , t^{3r-s}$).
A modified argument works for the general case, but we
can only find the leading term, \ie{} that
$$
\alpha_{2r}
\=\E (X_{a_1,\dots,a_m})^{2r}
={{ (2r)!} \over {2^r r!}}+O(t^{-1}) \quad .
$$
Of course the odd moments are all zero, since the distribution of
$\maam$ is symmetric about $\mu$.

In the two-letter case, the mean of $M_{a,b}$ is simply $ab/2$, so the
probability generating function for $\mab - \mu$ is, see \eqref{F},
$$
G\ab(q)={{F\ab(q)} \over { q^{ab/2} }}=
{{a!b!(1-q^{a+b})(1-q^{a+b-1}) \cdots (1-q^{a+1})}
\over
{q^{ab/2}(a+b)!(1-q^{b})(1-q^{b-1}) \cdots (1-q^{1})}} \quad.
$$

Taking ratios, we have:
\begin{equation}
  \label{ratio}
{{G\ab(q)} \over {G\awb(q)}}=
{ {a(1-q^{a+b})} \over {q^{b/2} (a+b)(1-q^a)} }
\quad.
\end{equation}

Recall that the binomial moments $\A_r:=\E[ {\binom{\mab-\mu}{r}}]$
are the Taylor coefficients of the probability generating function
(in our case $G\ab(q)$) around $q=1$.
Writing $q=1+z$, we have
$$
G\ab(1+z)=\sum_{r=0}^{\infty} \A_r(a,b) z^r \quad.
$$
Note that $\A_0(a,b)=1$ and $\A_1(a,b)=0$.
Let us call the expression on the right side of \eqref{ratio},
with $q$ replaced by $1+z$, $P(a,b,z)$:
$$
P(a,b,z):={ {a(1-(1+z)^{a+b})} \over {(1+z)^{b/2} (a+b)(1-(1+z)^a)} }
\quad.
$$
Maple can easily expand $P(a,b,z)$ to any desired power of $z$,
It starts out with
\begin{multline*}
P(a,b,z)=
1+\frac1{24} \left( 2\,a+b \right) b{z}^{2}-\frac1{24} \left( 2\,a+b \right) b
{z}^{3}
\\
-{\frac {1}{5760}} \left(8\,{a}^{3}-8\,{a}^{2}b-12\,a{b}^{2}-3\,{b}^{3}
-440\,a-220\,b \right) b{z}^{4} + \dots
\end{multline*}
note that the coefficients of all the powers of $z$ are polynomials
in $(a,b)$.

So let us write
$$
P(a,b,z)=\sum_{i=0}^{\infty} p_i(a,b)z^i \quad,
$$
where $p_i(a,b)$ are certain polynomials that Maple can compute
for any $i$, no matter how big.

Looking at the recurrence
$$
G\ab(1+z)=P(a,b,z)G\awb(1+z) \quad ,
$$
and comparing coefficients of $z^r$ on both sides, we get
\begin{equation}
\label{recurrence}
\A_r(a,b)-\A_r(a-1,b)=\sum_{s=1}^{r} \A_{r-s}(a-1,b) p_s(a,b)  
\quad .
\end{equation}
Assuming that we already know the polynomials
$\A_{r-1}(a,b),\A_{r-2}(a,b), \dots,\allowbreak \A_0(a,b)$, the left side
is a certain specific polynomial in $a$ and $b$, that Maple can easily
compute, and then $\A_r(a,b)$ is simply the indefinite
sum of that polynomial, that Maple can do just as easily.
So \eqref{recurrence} enables us to get {\it explicit} expressions
for the binomial moments $\A_r(a,b)$ for {\it any} (numeric) $r$.

But what about the general (symbolic) $r$? 
It is too much to hope for the full expression, 
{\bf but} we can easily conjecture as many leading terms
as we wish.
We first conjecture, and then immediately prove by induction, that for $r\ge1$
\begin{align*}
\A_{2r}(a,b)&={{1} \over {r!}} 
\left ( {{ab(a+b)} \over {24}} \right )^r + 
\text{lower order terms}
\\
\A_{2r+1}(a,b)&={{-1} \over {(r-1)!}} 
\left ( {{ab(a+b)} \over {24}} \right )^r + 
\text{lower order terms}
\quad,
\end{align*}
where we can conjecture (by fitting polynomials in $(a,b)$ to the
data obtained from the numerical $r$'s) any (finite, specific)
number of terms.

Once we have asymptotics, to any desired order, for the
binomial moments, we can easily compute the moments
$\mu_r(a,b)$ of $M_{a,b}-\mu$ themselves, for {\it any} desired
specific $r$ and 
asymptotically, to any desired order. We do that by using the
expressions of the powers as linear combination of falling-factorials
(or equivalently binomials) in terms of Stirling numbers of the
second kind, $S(n,k)$. Note that for the asymptotic expressions
to any desired order, we can still do it symbolically, since
for any specific $m$, $S(n,n-m)$ is a polynomial in $n$
(that Maple can easily compute, symbolically, as a polynomial in $n$).
In particular, the variance is:
$$
\sigma^2
=\mu_2(a,b)
={{ab(a+b+1)} \over {12}} \quad,
$$
in accordance with \eqref{sigma}.
In general we have $\mu_{2r+1}(a,b)=0$, of course, and
the six leading terms of 
$\mu_{2r}(at,bt)$ can be found in the webpage of this article.
{}From this, Maple finds that
$\alpha_{2r}(at,bt):=\mu_{2r}(at,bt)/\mu_2(at,bt)^r$ are
given asymptotically (for fixed $a,b$ and $t \rightarrow \infty$)
by:
$$
\alpha_{2r}(at,bt)=
{{(2r)!} \over {2^r r!}}  \cdot
\left( 1-{\frac {r (r-1)\left( {b}^{2}+ab+{a}^{2} \right) }
{5ab \left( a+b \right)}}\cdot {{1} \over {t}} \right)   
+O(t^{-2}) \quad .
$$
In particular, as $t \rightarrow \infty$, they converge to
the famous moments of $\cN(0,1)$. QED.

\subsection{The general case}

To merely prove asymptotic 
normality, one does not need a computer, since we only
need the leading terms. The above proof can be easily adapted
to the general case
$(\aam)=(ta^0_1,\dots,ta^0_m)$. One simply uses induction on $m$, the
number of different letters.

\subsection{The Maple package MahonianStat}

The Maple package {\tt MahonianStat}, accompanying this
article, has lots of features, that the readers can explore
at their leisure. Once downloaded into a directory,
one goes into a Maple session, and types
{\tt read MahonianStat;}. To get  a list of
the main procedures, type: {\tt ezra();}.
To get help with a specific procedure, type
{\tt ezra(ProcedureName);}. Let us just mention some
of the more important procedures.

{\tt AsyAlphaW2tS(r,a,b,t,s)}:
inputs symbols {\tt r,a,b,t} and a positive integer $s$,
and outputs the  asymptotic expansion, to order $s$,
for $\alpha_{2r}$ (=$\mu_{2r}/\mu_2^r$)

{\tt ithMomWktE(r,e,t)}: the $r$-th moment about the mean of the number
of inversions of $a_1t$ $1$'s, $\dots$, $a_mt$ $m$'s  in terms of the
elementary symmetric functions,
in $a_1, \dots , a_m$. Here $r$ is a specific (numeric)
positive integer, but $e$ and $t$ are symbolic.

{\tt AppxWk(L,x)}: Using the asymptotics implied by the
asymptotic normality of the (normalized) random variable
under consideration, finds an approximate value for
the number of words with $L[1]$ $1$'s,
$L[2]$ $2$'s, $\dots$, $L[m]$ $m$'s with exactly
$x$ inversions. For example, try:
{\tt AppxWk([100,100,100],15000);}

For the two-lettered case, one can get
better approximations, by procedure
{\tt BetterAppxW2}, that uses improved
limit-distributions, using more terms in
the probability density function.

The webpage of this article has some sample input and output.

\section{A general proof of \refT{T1}}\label{Spf2}

We have an exact formula \eqref{sigma} for the variance $\gs^2$ of $\maam$.
We first show that $\gs^2$ is always of the order $\Theta(N^2\Ax)$.

\begin{lemma}\label{Lsigma}
For any $\aam$,
  \begin{equation*}
\frac{N^2\Ax}{36}	
\le \gs^2
\le
\frac{(N+1)N\Ax}{12}	
\le
\frac{N^2\Ax}{6}.
  \end{equation*}
\end{lemma}

\begin{proof}
  For the upper bounds we assume, by symmetry, that $a_1\ge\dots\ge
  a_m$.
Then $\ax=a_1$ and
\begin{equation*}
  e_2=a_1\sum_{j=2}^m a_j+a_2\sum_{j=3}^m a_j+\dots
\le N \sum_{j=2}^m a_j = N\Ax.
\end{equation*}
Since $e_1=N$, \eqref{sigma} yields the upper bounds.

For the lower bound, we first observe that $2e_2e_1-6e_3\ge0$ (since
this difference can be written as a sum of certain $a_ja_ka_l$).
Hence $e_3\le e_1e_2/3$ and \eqref{sigma} yields
\begin{equation*}
  12\gs^2 \ge e_1e_2-e_3\ge \tfrac23 e_1e_2.
\end{equation*}
Further,
\begin{equation*}
  2e_2=\sum_{j=1}^ma_j(N-a_j)
\ge \sum_{j=1}^ma_j(N-\ax)
=N\Ax,
\end{equation*}
and the lower bound follows.
\end{proof}

\begin{proof}[Proof of \refT{T1}]
{}From  \eqref{F} follows the identity
\begin{equation}
  \label{f2q}
F_{n_1,n_2}(e^{\ii \theta}) ~~=~~
\prod_{j=1}^{n_2}
\frac{(e^{\ii (n_1+j)\theta}-1)/(\ii (n_1+j)\theta)}
     {(e^{\ii j\theta}-1)/(\ii j\theta)}.
\end{equation}
 By Taylor's series
$$
\log\frac{e^z-1}{z} = z/2 + z^2/24 + O(z^4), \qquad |z|\le 1,
$$
and 
we substitute this expansion into the identity \eqref{f2q}
to conclude:
\begin{multline}\label{r3}
F_{n_1,n_2}(e^{\ii \theta}) 
=
\exp\left(\ii  n_1 n_2 \theta/2 - n_1 n_2 (n_1+n_2+1)
          \theta^2/24 + O(n_2n_1^4\theta^4)
                            \right),  
\end{multline}
uniformly for $n_1\ge n_2\ge 1$ and $|\theta|\le(n_1+n_2)^{-1}$.

We use the factorization \eqref{F2}. By symmetry, we may assume
$a_1\ge a_2\ge\dots\ge a_m$, and then $A_{j-1}\ge a_{j-1}\ge a_j$ for
each $j$. Thus \eqref{r3} yields, uniformly for $q=e^{\ii\gth}$ with
$|\gth|\le N\qw$,
{\multlinegap=0pt 
\begin{multline*}
F_\aam(q)
 =\prod_{j=2}^m F_{A_{j-1},a_j}(q)
\\
=
\exp\left(\sum_{j=2}^m \Bigpar{\ii  A_{j-1} a_j \theta/2 
 - A_{j-1} a_j (A_{j}+1)\theta^2/24 + O(a_jA_{j-1}^4\theta^4)}
 \right).
\end{multline*}}
Here, the sums of the coefficients of $\gth$ and $\gth^2$ are easily
evaluated, but we do not have to do that since they have to equal
$\ii\mu$ and $-\gs^2/2$, respectively. Further,
\begin{equation}
  \label{pa1}
\sum_{j=2}^m A_{j-1}^4a_j \le N^4\sum_{j=2}^m a_j  =N^4\Ax.
\end{equation}
Consequently, if $|\gth|\le N\qw$,
\begin{equation}
  \label{pa2}
F_\aam(\eith)
=
\exp\bigpar{\ii\mu\gth-\gs^2\gth^2/2+O(N^4\Ax\gth^4)}
\end{equation}
and, by \eqref{G},
\begin{equation}
  \label{pag}
G_\aam(\eith)
=
\exp\bigpar{-\gs^2\gth^2/2+O(N^4\Ax\gth^4)}.
\end{equation}
Let $\gth=t/\gs$. For any fixed $t$, by \refL{Lsigma},
\begin{equation*}
  |Nt/\gs| = O\bigpar{\Ax\qqw}=o(1),
\end{equation*}
so $|\gth|\le N\qw$ if $\nu$ is large enough. Hence, by
\eqref{pag} and \refL{Lsigma},
\begin{equation*}
  \begin{split}
G_\aam\bigpar{e^{\ii t/\gs}}
&=
%e^{-\ii t\mu/\gs}F_\aam\bigpar{e^{\ii t/\gs}}
%\\&
\exp\lrpar{-\frac{t^2}2+O\Bigpar{{\frac{N^4\Ax t^4}{N^4\Ax^2}}}}
=\exp\bigpar{-t^2/2+o(1)},
 \end{split}
\end{equation*}
and \refT{T1} follows by the continuity theorem
\cite[Theorem XV.3.2]{FellerII}.
\end{proof}
%\qed

\section{The local limit theorem}\label{Slocal}

``If one can prove a central limit theorem for a sequence $a_n(k)$ of numbers arising
in enumeration, then one has a qualitative feel for their behavior.  A local limit
theorem is better because it provides asymptotic information about $a_n(k)$ $\dots$,''
\cite{Bender}.
In this section we prove that the relation
\eqref{local}
holds uniformly over certain very general, albeit not unrestricted, sets of
tuples $\bfa=(a_1,\dots,a_m)$.  The exact statement is given below in
\refT{Tlocal}.
% Henceforth, $m$ denotes the indicated minimum.

As explained in \citet{Bender}, there are two standard conditions for passage from a central
to a local limit theorem:
(1) if the sequence in question is unimodal, then one has a local limit theorem
for $n$ in the set $\{|n-\mu|\ge\epsilon\sigma\}$, $\epsilon>0$;  (2) if the
sequence in question is log-concave, then one has a local limit theorem for
all $n$.  Our sequence, the coefficients of the $q$-multinomial, is in fact
unimodal, as first shown by \citet{Schur} using invariant theory, and later by
\citet{Ohara} using combinatorics.  Unfortunately, the ensuing local limit theorem
fails to cover the most interesting coefficients, the largest ones, near the
mean $\mu$.  However, our polynomials are manifestly not log-concave as is seen
by inspecting the first three coefficients (assuming $n_1,n_2\ge 2$)
$$
\binom{n_1+n_2}{ n_1}_q = 1+q+2q^2+\cdots  \,\, .
$$
The question arises might the coefficients
be log-concave near the mean, and here is a small table of empirical values:
($c[j] = [q^j] \binom{2n}{n}_q$)
$$\vbox{
\halign{\hfill#\hfil & \quad\hfill#\hfil \cr
$n$&  $ (c[n^2/2-1])^2 - c[n^2/2] \times c[n^2/2-2]  $  \cr
2& -1 \cr
4& -7 \cr
6& -165 \cr
8& -1529 \cr
10& 44160 \cr
12& 7715737 \cr
14& 905559058 \cr
16& 101507214165 \cr
18& 11955335854893 \cr
20& 1501943866215277 \cr
}}$$

\vskip 5pt

\noindent Based on this scant evidence,
we speculate that some sort of log-concavity theorem is true,
but that its proper statement is complicated by describing the
appropriate range of $\bfa$ and $j$.  Thus, we use neither of the
two standard methods mentioned above for proving our local limit
theorem. 
(Later, we shall see that our theorem has implications for log-concavity.)
Instead, we use another standard method, direct integration
(Fourier inversion) of the characteristic function, or equivalently of
the probability generating function $F(q)$ for $q=\eith$ on the unit
circle.
We begin with one such estimate for rather small $\gth$.

\begin{lemma}
  \label{LG1}
There exists a constant $\tau>0$ such that for any $\aam$ and
$|\gth|\le\tau/N$,
\begin{equation*}
  \bigabs{F_\aam(\eith)}
=   \bigabs{G_\aam(\eith)}
\le e^{-\gss\gth^2/4}.
\end{equation*}
\end{lemma}

\begin{proof}
  Suppose that $0<\abs\gth\le\tau/N$. Then, using \refL{Lsigma},
  \begin{equation*}
	\frac{N^4\Ax\gth^4}{\gss\gth^2}
\le
	\frac{N^2\Ax\tau^2}{\gss} 
\le36\tau^2,
  \end{equation*}
so if $\tau$ is chosen small enough, the error term $O(N^4\Ax\gth^4)$
in \eqref{pa2} and \eqref{pag} 
is $\le\gss\gth^2/4$, and thus the result follows from \eqref{pag}.
\end{proof}

We let in the sequel $\tau$ denote this  constant. We may
assume $0<\tau\le1$.

\begin{lemma}
Uniformly, for all $\aam$ and all integers $k$,
 \label{Lmagnus}
 \begin{equation*}
\left| \P(\maam=k)
-
\frac1{\sqrt{2\pi}\gs}e^{-(k-\mu)^2/(2\gs^2)}
\right|
\le
\int_{\tau/N}^\pi|F_\aam(\eith)|\dd\gth
+O\Bigpar{\frac1{\gs\Ax}}.
  \end{equation*}
\end{lemma}

\begin{proof}
 For any integer $k$,
 \begin{equation*}
   \begin{split}
\P(&\maam=k)-
\frac1{\sqrt{2\pi}\gs}e^{-(k-\mu)^2/(2\gs^2)}
\\&=
\frac{1}{2\pi}\int_{-\pi}^\pi F_\aam(\eith)e^{-\ii k\gth}\dd\gth	 
-\frac1{2\pi}\intoooo e^{-\gss\gthh/2}e^{-\ii(k-\mu)\gth}\dd\gth
\\&=
\frac{1}{2\pi}\int_{\abs\gth\le\tau/N}
\Bigpar{G_\aam(\eith)-e^{\gss\gthh/2}}e^{-\ii (k-\mu)\gth}\dd\gth	 
\\&\qquad\qquad
+\frac{1}{2\pi}\int_{\tau/N\le\abs\gth\le\pi}
 F_\aam(\eith)e^{-\ii k\gth}\dd\gth	 
\\&\qquad\qquad
-\frac1{2\pi}\int_{\abs\gth\ge\tau/N}
 e^{-\gss\gthh/2}e^{-\ii(k-\mu)\gth}\dd\gth 
\\
&=: I_1+I_2+I_3. 
  \end{split}
 \end{equation*}

By \eqref{pag} and the inequality
$|e^w-1|\le|w|\max(1,|e^w|)$ we find for $\abs\gth\le \tau/N$, 
using \refL{LG1}, 
\begin{equation*}
  \begin{split}
  \bigabs{G_\aam(\eith)-e^{-\gss\gthh/2}}
&\le O(N^4\Ax\gth^4)\max\lrpar{e^{-\gss\gthh/2},|G_\aam(\eith)|}
\\&
= O\bigpar{N^4\Ax\gth^4 e^{-\gss\gthh/4}}.	
  \end{split}
\end{equation*}

Integrating, we find
\begin{equation*}
  \begin{split}
|I_1|
&\le
\int_{\abs\gth\le \tau/N}  \Bigabs{G_\aam(\eith)-e^{-\gss\gthh/2}} \dd\gth
\\
&
\le O\bigpar{N^4\Ax}\intoooo\gth^4 e^{-\gss\gthh/4}\dd\gth
=O\bigpar{N^4\Ax\gs^{-5}}
=O\Bigpar{\frac1{\gs\Ax}}.	
  \end{split}
\end{equation*}

Further, again using \refL{LG1},
\begin{equation*}
  \begin{split}
\abs{I_3}
\le
\int_{\tau/N}^\infty e^{-\gss\gthh/2}\dd\gth 
\le 3\gs\qw e^{-(\gs\tau/N)^2/2}
\le \frac{6}{\gs(\gs\tau/N)^2}
=O\Bigpar{\frac1{\gs\Ax}}.	
  \end{split}
\end{equation*}
Finally, 
$\abs{I_2}\le\int_{\tau/N}^\pi|F_\aam(\eith)|\dd\gth	 $.
\end{proof}

In order to verify \refConj{Conjlocal}, it thus suffices to show that
the integral $\int_{\tau/N}^\pi|F_\aam(\eith)|\dd\gth$ in \refL{Lmagnus}
is $O\Bigpar{\frac1{\gs\Ax}}$.

\begin{remark}
For example, an estimate 
\begin{equation}
  \label{conj3}
F_\aam(\eith)=O\Bigparfrac1{\gs^3\gth^{3}},
\qquad 0<\gth\le\pi,
\end{equation} 
is sufficient for \eqref{local}.
We conjecture that this estimate \eqref{conj3} holds when $\Ax\ge6$,
say. Note that it does not hold for very small $\Ax$: taking
$\gth=\pi$ we have, for even $n_1$,
$F_{n_1,1}(-1)=1/(n_1+1)=1/N$, and the same holds for $F_{n_1,2}(-1)$.

Note further that even the weaker estimate
\begin{equation}
  \label{conj2}
%F_\aam(\eith)=O(|\gs\gth|^{-2})
F_\aam(\eith)=O\Bigparfrac1{\gss\gth^{2}},
\qquad 0<\gth\le\pi,
\end{equation} 
would be enough to prove \eqref{local} with the weaker error term
$O(\Ax\qqw)$.  
\end{remark}

We obtain a partial proof of \refConj{Conjlocal} using the following lemma.

\begin{lemma}
  \label{LC1}
For a given $\tau\in(0,1]$
there exists $c=c(\tau)>0$ such that
\begin{equation}  \label{r2}
\abs{F_{n_1,n_2}(\eith)}
~\le~~
e^{-cn_2}
\end{equation}
for $n_1\ge n_2\ge 1$ and $\tau/(n_1+n_2)\le|\theta|\le\pi$.

More generally, for any $\aam$ and 
$\tau/N\le|\theta|\le\pi$,
\begin{equation}  \label{r2m}
\abs{F_{\aam}(\eith)}
~\le~~
e^{-c\Ax}.
\end{equation}
\end{lemma}

\begin{proof}
We prove first \eqref{r2}.
%The case $n_1=n_2=1$ can be treated
%separately, and then for the rest of our proof we may assume $n_1\ge2$.
%
For positive integer $n$ define
$$
f_{n}(y,q) = \prod_{j=0}^{n} (1-yq^j)^{-1}.
$$

For $0\le R<1$, we have (\eg{} by Taylor expansions) 
$e^{2R}\le \frac{1+R}{1-R}$,
and thus $e^{4R}\le \frac{(1+R)^2}{(1-R)^2}=1+\frac{4R}{(1-R)^2}$.
Hence, by convexity, for any real $\zeta$,
\begin{equation*}
  e^{2R(1-\cos\zeta)} 
\le
1+\frac{2R(1-\cos\zeta)}{(1-R)^2}
=
\frac{1+R^2-2R\cos\zeta}{(1-R)^2}
=\frac{|1-Re^{\ii\zeta}|^2}{(1-R)^2},
\end{equation*}
and thus
$$
\left|
(1-Re^{\ii \zeta})^{-1}
\right|
~\le~~
(1-R)^{-1} \, \exp\left(-R(1-\cos\zeta) \right).
$$
Consequently, by a simple trigonometric identity,
for any real $\phi$ and $\gth$,
\begin{align*}
\left |
f_{n_1}(R e^{\ii \phi},e^{\ii\theta})
\right |
~&\le~~
(1-R)^{-n_1-1} \,  \\
&\quad\times \, 
\exp\left(-R\Bigpar{n_1+1-\cos\bigpar{\phi+\frac{n_1}{2}\theta}
  \frac{\sin (n_1+1)\theta/2}{\sin\theta/2}}
                   \right)
\\
~&\le~~
(1-R)^{-n_1-1} \,\times \,
\exp\left(R\Bigpar{-n_1-1+
  \frac{\sin (n_1+1)\theta/2}{\sin\theta/2}}
                   \right).
\end{align*}
The function $g(\theta)=g_n(\gth)\=\frac{\sin n(\theta/2)}{\sin(\theta/2)}$,
where $n\ge1$, is an even function of $\theta$; 
is decreasing for $0\le\theta\le\pi/n$, as can be verified by
calculating $g'$; and
satisfies $|g(\theta)|\le g(\pi/n)$ for $\pi/n\le|\theta|\le\pi$.
Further, for  $n\ge 2$ and 
$0\le|\theta|\le\pi/n$,
\begin{equation*}
  \begin{split}
  g_n(\gth)
&=2\frac{\sin(n\gth/4)}{\sin(\gth/2)}\cos(n\gth/4)
=2g_{n/2}(\gth)\cos(n\gth/4)
\le n\cos(n\gth/4)
\\&
\le n\Bigpar{1-\frac{n^2\gthh}{40}}.	
  \end{split}
\end{equation*}
Let $\theta_0=\tau(n_1+n_2)^{-1}<\pi/(n_1+1)$.  
For $\gth_0\le|\gth|\le\pi$ we thus have
$$
|g_{n_1+1}(\gth)|
\le
g_{n_1+1}(\gth_0)
~\le~~
n_1+1 -\frac{n_1^3\theta_0^2}{40};
$$
whence, for $0\le R<1$, the estimate above yields
\begin{equation}\label{tho}
\left |
f_{n_1}(R e^{\ii \phi},e^{\ii\theta})
\right |
~\le~~
(1-R)^{-n_1-1}  \, \exp\left(
   -Rn_1^3\theta_0^2/40
                   \right).
\end{equation}

Combinatorially we know that $[y^{\ell}q^n]f_{n_1}(y,q)$ is the number
of partitions 
of $n$ having at most $\ell$ parts no one of which exceeds $n_1$.
As said in \refR{Rpart}, this equals $[q^n]\binom{n_1+\ell}{n_1}
F_{n_1+\ell}(q)$. Hence, using
Cauchy's integral formula, for any $R>0$,
\begin{equation*}
  \binom{n_1+n_2}{n_1}F_{n_1,n_2}(q)
=
[y^{n_2}] f_{n_1}(y,q)
=
\frac{1}{2\pi\ii}\int_{|y|=R} f_{n_1}(y,q)\frac{\dd y}{y^{n_2+1}}
\quad.
\end{equation*}
Consequently, \eqref{tho} implies that for $\gth_0\le|\gth|\le\pi$ and
$0<R<1$,
\begin{equation*}
  \binom{n_1+n_2}{n_1}\bigabs{F_{n_1,n_2}(q)}
\le
(1-R)^{-n_1-1} R^{-n_2}  \, \exp\left(
   -Rn_1^3\theta_0^2/40
                   \right).
\end{equation*}

Now choose $R=\rho:=n_2/(n_1+n_2)\le1/2$.
By Stirling's formula,
$$
\binom{n_1+n_2}{ n_1} 
= \Omega\bigpar{n_2^{-1/2}} \, (1-\rho)^{-n_1-1} \, \rho^{-n_2}
$$
and thus,
for $\gth_0\le|\gth|\le\pi$,
\begin{equation*}
\bigabs{F_{n_1,n_2}(q)}
\le
O(n_2\qq)  \, \exp\left(-\rho n_1^3\theta_0^2/40\right)
=
O(n_2\qq)  \, \exp\left(-\Omega(n_2)\right).
\end{equation*}
This shows \eqref{r2} for $n_2$ sufficiently large.  To handle the remaining
finitely many values of $n_2$ we shall show: for each $n_2\ge 1$ and
$\tau\in(0,1]$, there exists $\delta>0$ such that
\begin{equation} \label{r22}
\abs{F_{n_1,n_2}(\eith)}
~\le~~
1-\delta
\end{equation}
for all $n_1\ge n_2$ and $\tau/(n_1+n_2)\le|\gth|\le\pi$.
To do this, we use
\begin{equation}
  \label{aug}
  \begin{split}
\abs{F_{n_1,n_2}(\eith)}
~&=~~
\prod_{j=1}^{n_2}\frac{j}{n_1+j}\left|
                          \frac{\sin(n_1+j)\gth/2}{\sin (j\gth/2)}
                              \right|
\\
&=~~
\prod_{j=1}^{n_2}\frac{j\sin(\gth/2)}{\sin j\gth/2}
\cdot
\prod_{j=1}^{n_2}\frac{\bigabs{g_{n_1+j}(\gth)}}{n_1+j}
= \Pi_1\cdot \Pi_2
\quad.
  \end{split}
\end{equation}
Let $N=n_1+n_2$ and $\tau/N\le|\gth|\le\pi$.
For $n\ge n_1+1$ we have $|n\gth|\ge n_1|\gth|\ge N|\gth|/2\ge\tau/2$,
and thus the estimates above show that
\begin{equation*}
|g_n(\gth)|\le g_n(\tau/2n)  \le n(1-\tau^2/160)\quad.
\end{equation*}
Hence the final product $\Pi_2$ in \eqref{aug} is bounded by $1-\tau^2/160<1$.
The product $\Pi_1$ is a continuous function of $\gth$, and
equals 1 for $\gth=0$;
hence $|\Pi_1|\le 1+\tau^2/200$ for $|\gth|\le\eps$, where $\eps>0$ is
sufficiently small. (Recall that $n_2$ now is fixed.) 
This proves \eqref{r22} for $|\gth|\le\eps$.

For larger $|\gth|$ we use the factorization
\begin{equation}\label{sof}
  F_{n_1,n_2}(q)
=\binom{n_1+n_2}{n_2}\qw
\frac{(1-q^{n_1+1})\dotsm(1-q^{n_1+n_2})}{(1-q)\dotsm(1-q^{n_2})}.
\end{equation}
Let $0<|\gth_0|\le\pi$, and suppose that $k\ge0$ of the factors in the
denominator of \eqref{sof} vanish at $q=q_0:=e^{\ii\gth_0}$. Then
$0\le k\le n_2-1$, since $1-q_0\neq0$. There are at least $k$ factors
in the numerator of \eqref{sof} that vanish at $q_0$ (since $F$ is a
polynomial, and all factors have simple roots only); for
$q=e^{\ii\gth}$, each of these factors is bounded by $N|q-q_0|\le
N|\gth-\gth_0|$ while every factor is bounded by 2; hence the
numerator of \eqref{sof} is $O(N^k|\gth-\gth_0|^k)$. Let $J$ be an
interval around $\gth_0$ such that the denominator of \eqref{sof} does
not vanish at any $q=e^{\ii\gth}\neq q_0$ with $\gth\in\bar J$; then
the denominator is $\Theta(|\gth-\gth_0|^k)$ for 
% $q=e^{\ii\gth}$ with 
$\gth\in J$.
Finally, the binomial coefficient in \eqref{sof} is
$\Theta(n_1^{n_2})$.

Combining these estimates, we see that uniformly for $\gth\in J$,
\begin{equation*}
\lrabs{  F_{n_1,n_2}(e^{\ii\gth})}  
= O \Bigpar{\frac{N^k|\gth-\gth_0|^k}{n_1^{n_2}|\gth-\gth_0|^k}}
=O\bigpar{n_1^{k-n_2}}
=O\bigpar{n_1^{-1}}.
\end{equation*}
Since the set $\eps\le|\gth|\le\pi$ may by covered by a finite number
of such interval $J$, 
$\lrabs{  F_{n_1,n_2}(e^{\ii\gth})}  =O\bigpar{n_1^{-1}}$ uniformly for
 $\eps\le|\gth|\le\pi$. Consequently \eqref{r22} holds for all such
$\gth$ if $n_1$ is sufficiently large.

It remains to verify \eqref{r22} for each fixed $n_2$ and a finite
number of $n_1$; in other words, that for each $n_2\ge1$ and $n_1\ge
n_2$, there exists $\gd>0$ such that \eqref{r22} holds.
To see this, note that the events $M_{n_1,n_2}=0$ and $M_{n_1,n_2}=1$
both have positive probability. It follows that 
$\abs{F_{n_1,n_2}(\eith)}<1$ for every $\gth$ with $0<|\gth|\le\pi$,
and \eqref{r22} follows.
This completes the proof of \eqref{r2}.

To prove \eqref{r2m}, we assume as we may that $a_1\ge\dots\ge a_m$
and use the factorization \eqref{F2}.
Let $J$ be the first index such that $a_2+\dots+a_J\ge\Ax/2$.
For $j\ge J$, then $A_{j-1}+a_j=A_j\ge A_J\ge a_1+\Ax/2\ge N/2$, and
thus $A_j|\gth|\ge N|\gth|/2\ge\tau/2$; hence \eqref{r2} yields
\begin{equation*}
 \bigabs{F_{A_{j-1},a_j}(\eith)} \le e^{-c(\tau/2) a_j}. 
\end{equation*}
We thus obtain from \eqref{F2}, since each $F_{n_1,n_2}$ is a
probability generating function and thus is bounded by 1 on the unit circle,
\begin{equation*}
|F_\aam(\eith)|=\prod_{j=2}^m |F_{A_{j-1},a_j}(\eith)|
\le \prod_{j=J}^m e^{-c(\tau/2) a_j}
%= e^{-c(\tau/2) \sum_{j=J}^m a_j}
\le  e^{-c(\tau/2) \Ax/2},
\end{equation*}
because $\sum_{j=J}^m a_j\ge \Ax/2$.
This proves \eqref{r2m} (redefining  $c(\tau)$).
\end{proof}

\begin{theorem}
  \label{Tlocal}
There exists a positive constant $c$
such that for every $C$,
the following is true.  
Uniformly for all $\aam$ such that $\ax\le C e^{c\Ax}$ and all integers $k$,
  \begin{equation}\label{localxxx}
\P(\maam=k)
=\frac1{\sqrt{2\pi}\gs}\Bigpar{e^{-(k-\mu)^2/(2\gs^2)}+O\Bigparfrac1{\Ax}}.	
  \end{equation}
\end{theorem}

\begin{proof}
Let $c_1=c(\tau)$ be the constant in \refL{LC1}.
Then, Lemmas \refand{Lmagnus}{LC1} yield
    \begin{equation*}%\label{localxxx}
\P(\maam=k)
=\frac1{\sqrt{2\pi}\gs}\Bigpar{e^{-(k-\mu)^2/(2\gs^2)}
+O\Bigpar{\frac1{\Ax}+\gs e^{-c_1\Ax}}}.	
  \end{equation*}
For any fixed $c<c_1$ we have, using \refL{Lsigma}, 
$\gs\Ax e^{-c_1\Ax}=O(Ne^{-c\Ax})$
and thus
    \begin{equation*}
\P(\maam=k)
=\frac1{\sqrt{2\pi}\gs}\Bigpar{e^{-(k-\mu)^2/(2\gs^2)}
+O\Bigpar{\frac{1+N e^{-c\Ax}}{\Ax}}}.	
  \end{equation*}
The result follows, since $N e^{-c\Ax}=\ax e^{-c\Ax}+\Ax e^{-c\Ax}
=\ax e^{-c\Ax}+O(1)$.
\end{proof}

\subsection{Log-concavity}

Let us review the proof of \refT{Tlocal} with the intention of greater
accuracy.  The goal is to prove log-concavity in some range.
For concreteness, let $\bfa=(n,n)$.  Then
$\sigma^2$ is of order $n^3$, and for sufficient
accuracy we take the Taylor series in the exponent of \eqref{r3} out
to $O(\theta^{10})$.  This yields, for some polynomials $p_k(n)$ of
degree $k+1$,
\begin{equation*}
  \begin{split}
F_{n,n}(e^{\ii\theta})
&=
\exp\bigpar{\ii\mu\gth-\gss\gth^2/2+p_4(n)\gth^4+p_6(n)\gth^6+p_8(n)\gth^8
 +O(n^{11}\gth^{10})}
\\
&=
e^{\ii\mu\gth-\gss\gth^2/2}
\bigpar{1+p_4(n)\gth^4+p_6(n)\gth^6+p_8(n)\gth^8
\\&\hskip6em
+\tfrac12p_4^2(n)\gth^8+p_4(n)p_6(n)\gth^{10}+\tfrac16p_4^3(n)\gth^{12}
 +O(n^{11}\gth^{10})}
  \end{split}
\end{equation*}
Arguing as in the proof of \refL{Lmagnus} but using this estimate
instead of \eqref{pag} for $|\gth|\le \tau/N$, one easily obtains,
after the substitution $\gth=t/\gs$, for any $k$ and with $x\=(k-\mu)/\gs$,
\begin{equation*}
  \begin{split}
\P(M_{n,n}=k)=\frac1{2\pi}\intoooo e^{-t^2/2-\ii tx}
\Bigpar{1+\frac{p_4(n)}{\gs^4}t^4+\dots+\frac{p_4^3(n)}{6\gs^{12}}t^{12}} 
\frac{\dd t}{\gs} 
+ O(n^{-4}\gs\qw).
  \end{split}
\end{equation*}
Letting $\gf(x)\=(2\pi)\qqw e^{-x^2/2}$ denote the normal density
function, and $\gfd{j}$ its derivatives, we obtain by Fourier inversion
\begin{equation}\label{pmnn}
  \begin{split}
\P(M_{n,n}=k)
&=
\gs\qw
\Bigpar{\gf(x)+\frac{p_4(n)}{\gs^4}\gfd4(x)+\dots+\frac{p_4^3(n)}{6\gs^{12}}
\gfd{12}(x)+ O(n^{-4})}
\\
&=
\frac1{\sqrt{2\pi}\gs} e^{-x^2/2}\bigpar{1+Q(n,x)+O(n^{-4})},
  \end{split}
\raisetag{20pt}
\end{equation}
where $Q(n,x)$ is $\gs^{-12}$ times a certain polynomial  in $n$ and
$x$ of degree 17 in $n$; thus for $x=O(1)$ we have 
$Q(n,x)=O(n\qw)$ 
and similarly, for derivatives with respect to $x$,
$Q'(n,x)=O(n\qw)$ and $Q''(n,x)=O(n\qw)$.
($Q(n,x)$ can easily be computed explicitly using computer algebra,
but we do not have to do it.)

Replacing $k$ by $k\pm1$ in \eqref{pmnn} we find,
for $x=O(1)$,
\begin{equation*}
  \begin{split}
\P(M_{n,n}=k\pm1)
&=
\frac1{\sqrt{2\pi}\gs}
e^{-(x\pm\gs\qw)^2/2}\bigpar{
 1+Q(n,x)\pm\gs\qw Q'(n,x)+O(n^{-4})},
  \end{split}
\end{equation*}
and thus
\begin{equation*}
  \begin{split}
\P(M_{n,n}&=k-1) \P(M_{n,n}=k+1)
\\
&=
\frac1{{2\pi}\gss} 
e^{-x^2-\gs\qww}\bigpar{
(1+Q(n,x))^2-\gs\qww Q'(n,x)^2+O(n^{-4})}.
\\
&=
e^{-\gs\qww}\P(M_{n,n}=k)^2\bigpar{1+O(n^{-4})}.
  \end{split}
\end{equation*}
Hence, for $x=O(1)$, 
i.e., $k=\mu+O(\gs)$,
%when $\P(M_{n,n}=k)=\Theta(\gs\qw)=\Theta(n^{-3/2})$,
\begin{equation}\label{sjw}
  \begin{split}
\P(M_{n,n}=k)^2
&-
\P(M_{n,n}=k-1) \P(M_{n,n}=k+1)
\\
&=
\bigpar{\gs\qww+O(n^{-4})}\P(M_{n,n}=k)^2
%\\&
=
\frac1{{2\pi}\gs^4} 
e^{-x^2}\bigpar{1+O(n^{-1})}.
  \end{split}
\end{equation}

In particular, this is positive for large $n$.
This gives:

\begin{theorem}[A log-concavity result]
  \label{TC2}
For each constant $C$ we have $n_0$ such that
for $n\ge n_0$ and $|j-\mu| \le C \sigma$
$$
c_j^2 \ge c_{j-1} c_{j+1},
$$
where
$$
c_j \= [q^j] \binom{2n}{ n}_q = \binom{2n}{ n}\P(M_{n,n}=j).
$$
\end{theorem}

We note that the ``mysterious''
numbers appearing in our earlier table for the choice $j=n^2/2-1$
are asymptotically
$$
 \frac{1}{2\pi\sigma^4} \, \binom{2n}{n}^2
\sim
\frac{18}{\pi n^6} \, \binom{2n}{n}^2
\sim
\frac{18}{\pi^2} n^{-7}2^{4n}.
$$

\begin{remark}
  This argument for log-concavity in the central region does not use
  any special properties of the distribution; although we needed
  several terms in the asymptotic expansion above, it was only to see
  that they are sufficiently smooth, and the main term in the final
  result \eqref{sjw} comes from the main term
  $e^{-x^2/2}/(\sqrt{2\pi}\gs)$ in \eqref{localxxx}.
What we have shown is just that the convergence 
to the log-concave Gaussian function 
in the local limit  theorem is sufficiently regular for the
  log-concavity of the limit to transfer to $\P(M_{n,n}=k)$ for
  $k=\mu+O(\gs)$ and sufficiently large $n$.
\end{remark}

\section{Final comments}\label{Sfine}

Suppose that $\Ax\not\to\infty$. 
We may, as usual, assume that $a_1\ge\dots\ge a_m$.
By considering a subsequence (if
necessary), we may assume that $\Ax:=N-\ax=a_2+\dots+a_m$ is a constant; 
this entails that $m$ is bounded, so by again considering a
subsequence, we may assume that $m$ and $a_2,\dots,a_m$ are constant.
We thus study the case when $a_1\to\infty$ with fixed $a_2,\dots,a_m$.

In this case, the number of inversions between indices $2,\dots,m$ is
$O(1)$, which is asymptotically negligible. Ignoring these, we can
thus consider the random word
as $\Ax$ letters $2,\dots,m$ inserted in $a_1$ 1's, and the number of
inversions is the sum of their positions, counted from the end.
It follows easily, either probabilistically or by calculating the
characteristic function from \eqref{F}, that $\maam/N$, or
equivalently $\maam/a_1$, converges in distribution to the sum
$\sum_{j=1}^{\Ax} U_j$ of $\Ax$ independent random variables $U_j$ with
the uniform distribution on $[0,1]$. 
Equivalently, since $\gs^2\sim n_1^2\Ax/12\sim N^2\Ax/12$,
\begin{equation*}
\frac{\maam-\mu(a_1, \dots, a_m) }{\sigma(a_1, \dots , a_m)} \dto
\sqrt{\frac{12}{\Ax}} \sum_{j=1}^{\Ax} (U_j-\tfrac12),
\end{equation*}
where $\dto$ denotes convergence in distribution.
This limit is clearly not normal for any finite $\Ax$. 
(However, its distribution is close to standard normal for large $\Ax$.
Note that it is normalized to mean 0 and variance 1.)

\newcommand\AAP{\emph{Adv. Appl. Probab.} }
\newcommand\JAP{\emph{J. Appl. Probab.} }
\newcommand\JAMS{\emph{J. \AMS} }
\newcommand\MAMS{\emph{Memoirs \AMS} }
\newcommand\PAMS{\emph{Proc. \AMS} }
\newcommand\TAMS{\emph{Trans. \AMS} }
\newcommand\AnnMS{\emph{Ann. Math. Statist.} }
\newcommand\AnnPr{\emph{Ann. Probab.} }
\newcommand\CPC{\emph{Combin. Probab. Comput.} }
\newcommand\JMAA{\emph{J. Math. Anal. Appl.} }
\newcommand\RSA{\emph{Random Struct. Alg.} }
\newcommand\ZW{\emph{Z. Wahrsch. Verw. Gebiete} }
\newcommand\DMTCS{\jour{Discr. Math. Theor. Comput. Sci.} }

\newcommand\AMS{Amer. Math. Soc.}
\newcommand\Springer{Springer-Verlag}
\newcommand\Wiley{Wiley}

\newcommand\vol{\textbf}
\newcommand\jour{\emph}
\newcommand\book{\emph}
\newcommand\inbook{\emph}
\def\no#1#2,{\unskip#2, no. #1,} %(typeset after year) 
\newcommand\toappear{\unskip, to appear}

\newcommand\webcite[1]{%\hfil  %???
   %\penalty0 %???
\texttt{\def~{{\tiny$\sim$}}#1}\hfill\hfill}
\newcommand\webcitesvante{\webcite{http://www.math.uu.se/~svante/papers/}}
\newcommand\arxiv[1]{\webcite{arXiv:#1.}}

\def\nobibitem#1\par{}


\begin{thebibliography}{99}

\bibitem[Andrews(1976)]{Andrews} 
G. E. Andrews, 
{\it The Theory of Partitions}, Addison-Wesley, 
Reading, Mass., 1976.

\bibitem[Bender(1973)]{Bender} 
E. A. Bender, 
Central and local limit theorems applied to asymptotic enumeration, 
{\it J. Combinatorial Theory Ser. A} {\bf 15} (1973) 91--111.


\bibitem{FellerI}
W. Feller, 
\emph{An Introduction to Probability Theory and Its Application}, volume I, 
third edition, Wiley, %John Wiley \& Sons, 
New York, 1968.

\bibitem{FellerII}
W. Feller, 
\emph{An Introduction to Probability Theory and its Applications},
volume II, 2nd ed., Wiley, New York, 1971.

\bibitem[O'Hara(1990)]{Ohara}
K. M. O'Hara,
Unimodality of Gaussian coefficients: a constructive proof,
{\it Journal of Combinatorial Theory, Series A} 
{\bf 53} (1990) 29--52.

\bibitem[Schur(1968)]{Schur}
I. Schur, 
\emph{Vorlesungen \"uber Invariantentheorie}, % (German)
edited by %Bearbeitet und herausgegeben von 
H. %elmut 
Grunsky, 
%Die Grundlehren der mathematischen Wissenschaften, Band 143 
Springer-Verlag, Berlin, 1968.

\end{thebibliography}
\end{document}